\documentclass[a4paper,12pt]{amsart}
\usepackage[utf8]{inputenc}
\usepackage{mathtools}

\makeatletter
\@namedef{subjclassname@2020}{\textup{2020} Mathematics Subject Classification}
\makeatother

\allowdisplaybreaks[1]

\usepackage{amsmath,amsthm,amssymb}
\usepackage{hyperref}
\usepackage{amsfonts}
\usepackage{amsmath}
\usepackage{mathrsfs}
\usepackage{tabto}
\usepackage{changepage} 
\usepackage{fullpage} 
\usepackage{graphicx}
\usepackage{enumitem} 

\newcommand{\la}{\lambda}

\newcommand{\ity}{\infty}

\newcommand{\N}{\mathbb{N}}

\numberwithin{equation}{section}
\newtheorem{theorem}{Theorem}[section]
\newtheorem{lemma}[theorem]{Lemma}

\newtheorem{proposition}[theorem]{Proposition}

\newtheorem{remark}[theorem]{Remark}
\newtheorem{example}[theorem]{Example}
\newtheorem{definition}[theorem]{Definition}

\thanks {The research work of the first author is supported by ANRF(SERB) research grant TAR/2023/000197}

\begin{document}

\title[dynamics of bungee set of composite entire function]{results on dynamics of bungee set of composite entire functions in the eremenko-lyubich class}

\author[D. Kumar]{Dinesh Kumar}
\address{Department of Mathematics, Birla Institute of Technology Mesra, Ranchi, 
Jharkhand--835 215, India}
\email{dineshkumar@bitmesra.ac.in }

\author[S. Das]{Soumyajeet Das}
\address{Department of Physics, Birla Institute of Technology Mesra, Ranchi, 
Jharkhand--835 215, India}
\email{soumyajeetdas39@gmail.com }



\begin{abstract}
In this paper, we have discussed the dynamics of composite entire functions in terms of relationship between bungee set, escaping set and filled-in Julia set. We have established some relation between the dynamics of  composition of entire functions and the functions taken for composition. We have shown that the union of the bungee set of two entire functions contains the bungee set of the composite function. In addition, it is shown that the filled-in Julia  set of composite entire functions contains the filled-in Julia set of functions used for the composition. The results have been illustrated with several examples. We have mostly dealt with  permutable(commuting) functions.
\end{abstract}

\keywords{bungee set, wandering domain, Baker domain, repelling periodic point, rationally indifferent fixed point, escaping set, asymptotic value, singular set}

\subjclass[2020] {37F10,  30D05}

\maketitle

\section{Introduction}
When we iterate a holomorphic function $ f:\mathbb{C}  \rightarrow \mathbb{C} $, 
$n-$times, we observe that the points of the complex plane move differently, which we call the dynamics of the point. Here,  $f^n$  denotes the $n-$th iteration of  function  $f$, where $n \in \mathbb{N}$. As, not all  points have same dynamics, we therefore have different classifications of points of the complex plane $\mathbb{C}$; one classification is based on the nature of neighbourhood of a point $z_0$ (say); they are said to be in the Fatou set $F(f)$  (set of points whose neighbourhood  exhibits stable behaviour) and Julia set $J(f)$ (set of points whose neighbourhood exhibits chaotic behaviour). This was  broadly the classification we generally see when we talk about dynamics of points in complex plane. 

There is yet another classification of these points based on the nature of their orbits; viz.,  the filled-in  Julia set $K(f)$ (set of points whose orbits are bounded), escaping set $I(f)$ (set of points whose orbits tend to infinity), and bungee set $BU(f)$ (the orbits of such points neither remain bounded nor tend to infinity; only there exists a subsequence that remains bounded and another one that is unbounded). These sets  are disjoint among themselves and their union $BU(f)\cup I(f) \cup K(f)=\mathbb{C}$.  In \cite{osb2}, it was shown that if a Fatou component intersects with $BU(f)$ then that component becomes a wandering domain contained in $BU(f)$. 
It is to be noted that the bungee set and escaping set must have a non empty intersection with $J(f)$ (see \cite{e1, osb2}). The escaping set too has a non empty intersection with the Fatou set when the Fatou component is a Baker domain \cite{d1}, or when the Fatou set has a multiply connected wandering domain \cite{rs1}. Knowing all these sets for an entire function, we were curious about what happens when we compose two such functions and how these sets relate to the composition with those of individual functions. 
As composition of two  transcendental functions is again transcendental, we have tried to analyse what happens to the dynamics of transcendental functions under compositions.
\\ 
Simplifying the situation further, we have taken transcendental functions $f ,g$  in the Eremenko-Lyubich class $\mathcal{B}$ which means that the functions of this class have a bounded set of singular values (i.e., the set of all  asymptotic values and critical values and their finite limit points is bounded). Recall the notion of an oscillatory wandering domain.
\begin{definition}
 Given a transcendental entire function $f$, a domain $U\subset F(f)$ is called an oscillatory wandering domain of $f$, if  $U\subseteq BU(f)$, that is, $\{\infty, d\}$ is a limit function of $\{f^n\}$ on $U$ for some $d\in \mathbb{C}$, \cite{marti2020wandering}.
\end{definition}
 Additionally, class $\mathcal{B}$ may have an intersection of the Fatou set with the bungee set, in other words there can be oscillating wandering domain which also lies inside the bungee set for $f\in\mathcal{B}.$ 
As our aim was to establish a relationship between different sets of functions and their compositions,  it would be very difficult to establish any such relation in the absence of some algebraic relation among these functions. Hence, we have focused on permutable (commuting) functions, that is, functions $f, g$ satisfying $f \circ g =g \circ f.$ Also, we have considered entire functions $f$ and $g$ related by $g = f+ C$ where $C$ is the period of $f$. 
The following results are proven in this paper:\\
Let $f,g \in \mathcal{B}$ be commuting functions. The bungee set of composed function  is contained in the union of the bungee set of functions used for composition. 
We know some results for exponential functions which belong to Speiser class $\mathcal{S}$ (such class of functions has finite number of elements in their singular set) \cite{sch}. Also, $\mathcal{S} \subset \mathcal{B}$. It was shown in \cite{baker2} that for $f= e^{\lambda z}, \la\neq 0$, a function $g$ that commutes  with $f$ is some finite iteration of $f$ itself, that is, $g=f^n$ for some $n\in\N$. Hence, 
in this case, the result clearly holds, in fact, we obtain $BU(f)=BU(g)=BU(f\circ g).$

\subsection{Observations}
While developing the relationships among the sets of composite functions and the individual functions, we also found
that the bungee set and escaping set exhibit similar characteristics for functions of class $\mathcal{B}$:
\begin{enumerate}
\item The closure of both these sets equals Julia set.
\item Both sets are neither open nor closed \eqref{Bug_nop_nclo}.
\item Both sets of composite function is contained in union of individual functions (to be proved later on).
\end{enumerate}

\section{Main results}

\textbf{Property A} : $f$ (respectively $g$) has the property that it sends orbits escaping under $g$ (respectively, $f$) to $\ity.$\\

To begin with the commuting entire functions, we first understand the nature of these sets when functions satisfies \textbf{Property A}.

\begin{theorem} \label{Complete_invar_esp}
If $f$ and $g$ are commuting functions satisfying \textbf{Property A}, then $I(f)$ is completely invariant under $g$ and vice versa. 
\end{theorem}
\begin{proof}
Suppose $z_{0} \in I(f).$
Then, $|f^{n}(z_{0})| \rightarrow \infty$.
Now $|g(f^{n}(z_{0})|$ can either tend to $\infty$ or be bounded. Using  \textbf{Property A}, $|g(f^{n}(z_{0})| \rightarrow \infty $.
As $f$ and $g$ are commuting so, $|g(f^{n}(z_{0}))|$ = $|f^{n}(g(z_{0}))| \rightarrow \infty.$
This implies that  $I(f)$ is forward invariant under $g$.
It is easy to see that $g^{-1}(I(f))\subset I(f)$. 
Hence,  $I(f)$ is completely invariant under $g$. As $f$ and $g$ are commuting, their roles can be interchanged and hence, $I(g)$ is completely invariant under $f$ as well. 
\end{proof}

\begin{remark}
It follows from above result that   $I(f)$ and $I(g)$ are completely invariant under  $f \circ g$ as well.
\end{remark}
%

\begin{theorem} \label{K_BU_no_asymp} 
Suppose $f$ and $g$ are commuting entire functions satisfying \textbf{Property A}. Then $K(f)$ and   $BU(f)$ are completely invariant under $g$ and vice versa.
\end{theorem}
\begin{proof}
$K(f)$ is in general forward invariant under $g,$ because if we take a point $w\in K(f)$, then we have $|f^{n}(w)|< R$ for some $R.$ It is clear that $|g(f ^{n}(w))|=|f ^{n}(g(w))|< R_1$ for some $R_1$  which implies that $g(w)\in K(f)$ that proves the result. Now, let $z_0 \in g^{-1}(K(f)).$ This implies that  $|f^{n}(g(z_0))| < L$ for some $L$ and because of commutation, $|g(f^{n}(z_0))|< L.$ Using  \textbf{Property A},  $|f^{n}(z_0)|<L_1$ for some $L_1$, which means $z_0 \in K(f)$. Thus, we obtain that $K(f)$ is completely invariant under $g$. Similarly, $K(g)$ is also completely invariant under $f$ as well.
\end{proof}

\begin{theorem} \label{K_BU_no_asymp1} 
Suppose $f$ and $g$ are commuting entire functions satisfying \textbf{Property A}. Then  $BU(f)$ is completely invariant under $g$ and vice versa.
\end{theorem}
\begin{proof}
Using the dynamical partition of the complex plane, we obtain $BU(f)= \mathbb{C} \backslash (I(f))\cup K(f)).$ We know from preceding argument that $K(f)$ is completely invariant under $g$ and from Theorem \ref{Complete_invar_esp}, $I(f)$ is completely invariant under $g.$ Hence, on combining these results, we  conclude that $BU(f)$ is also completely invariant under $g$. On similar lines, $BU(g)$ will  also be completely invariant under $f$. 
\end{proof}

\begin{remark}
It follows from above result that   $I(f),$ $I(g); K(f), K(g);$ and $BU(f), BU(g)$ are completely invariant under  $f \circ g$ as well.
\end{remark}

%

The next result provides a relation between the filled-in Julia set of composite entire functions and the filled-in Julia set of the individual functions.
\begin{lemma}\label{lem2.3}
The filled-in Julia set of the composite function $f\circ g$ contains the intersection of the filled-in Julia set of the individual functions which are commuting  and satisfying \textbf{Property A}.
\end{lemma}

\begin{proof}
Suppose that $z_0\in K(f)\cap K(g)$. Consider the array of  following sequences:\\
$\begin{array}{ccccc}
f(g(z_0)) &f^2(g(z_0))&f^3(g(z_0))\cdots\\
f(g^2(z_0))&f^2(g^2(z_0))&f^3(g^2(z_0))\cdots\\
\vdots&\vdots&\vdots\\
f(g^k(z_0))&f^2(g^k(z_0))&f^3(g^k(z_0))\cdots\\
\vdots&\vdots&\vdots

\end{array}$\\
Now, the first row of the array will converge to say, $l$ for, if it does not, then $g(z_0)\in I(f)$ which implies that $z_0\in I(f)$, a contradiction. Using continuity of $g$, we can say that second row will converge to $g(l)$ and so on. Let $X$ be the set containing all these limit points. The following are the possible cases for set $X$.\\
Case 1: $l$ is periodic point of $g$. In particular, suppose that $l$ is a fixed point of $g$. Hence, every sequence of the array will be bounded which shows that $z_0\in K(f\circ g)$.\\
Case 2: $X$ is infinite and bounded. Then, by Bolzano-Weierstrass theorem, $X$ has a limit point which also shows that $z_0\in K(f\circ g)$.\\
Case 3: $X$ is unbounded, so that there exists a subsequence which  escapes to infinity. This in turn implies that $z_0\in I(f\circ g)\subset I(f) \cup I(g)$ (see \cite{d1}) which leads to  a contradiction.\\
From all the above cases, we obtain that  $K(f)\cap K(g)\subset K(f\circ g).$
\end{proof}
 
The next result connects bungee set of composite entire function with that of individual functions.
\begin{theorem}
Suppose $f$ and $g$ are commuting transcendental entire functions  satisfying \textbf{Property A}. Then, $BU(f \circ g)\subset BU(f)\cap BU(g).$
\end{theorem}

\begin{proof}


 Suppose $z_0\in BU(f\circ g).$ Then,  there exists two subsequences $\{n_k\},\{m_k\}$ and a constant $M>0$ such that $(f\circ g)^{n_k}(z_0)\to\infty$ as $k\to\infty$ and $|(f\circ g)^{m_k}(z_0)|<M \mbox{ for all } k\in \mathbb{N}$. It is enough to show that $z_0\in BU(g)$ (as the proof of $z_0\in BU(f)$ follows on similar lines). In order to show  this, we prove that $z_0$ cannot be in $I(g)$ as well as $K(g).$ We divide the proof into several cases:\\
Case(i): Suppose that $g^{n_k}(z_0)\to\infty$ and $g^{m_k}(z_0)\to\infty$ as $k\to\infty$.\\
We can assume that $g^n(z_0)\to\infty$ as $n\to\infty$
(for, if there exists a subsequence $\{l_k\}$ such 
that $g^{l_k}(z_0)$ stays bounded as $k\to\infty$ then 
this will clearly imply that $z_0\in BU(g)$). Also, using Theorem \ref{Complete_invar_esp}, $f^k(I(g)\subset I(g)$ for every  $k\in\mathbb{N}$. In particular, $f^{m_k}(z_0)\in I(g)$ which implies that $(g\circ f)^{m_k}(z_0)\to\infty $ as $k\to\infty$ which leads to a  contradiction. \\
Case(ii): Suppose that $g^{n_k}(z_0)\to b$ and $g^{m_k}(z_0)\to\infty$ as $k\to\infty$. This again shows that $z_0\in BU(g)$ and the result follows.\\
Case(3): Suppose that $g^{n_k}(z_0)\to \alpha $ and $g^{m_k}(z_0)\to \beta$ as $k\to\infty.$   It follows that $z_0\in K(g)$ (for,  if there exists a subsequence $\{l_k\}$ for 
which $g^{l_k}(z_0)\to\infty$ as $k\to\infty$ then we are back to Case(2) 
 and hence $z_0\in BU(g)$). Using Theorem \ref{Complete_invar_esp},  $f(K(g)\subset K(g)$. In particular, $f^{n_k}(z_0)\in K(g)$, i.e., there exists $L>0$ such that $|g^n(f^{n_k}(z_0))|<L \mbox{ for all } n\in\mathbb{N}$. This implies that $|(f\circ g)^{n_k}(z_0)|<L \mbox{ for all } k\in\mathbb{N}$ which is again a contradiction. Thus, from the above cases, we conclude that $z_0\in BU(g).$  The proof of $z_0\in BU(f)$ follows on similar lines. 
Hence, we obtain that $BU(f\circ g)\subset BU(f)\cap BU(g)\subset BU(f)\cup BU(g)$.
\end{proof}
Given the preceding proof, we ask a question:\\
\textit{\textbf{Question:}}  What can we say about points of the set $K(f) \cup K(g)\setminus K(f)\cap K(g),$ where  will it be situated?\\
  Let $w_0 \in K(f) \cup K(g) \ \backslash \ K(f) \cap K(g)$. These points won't intersect with $I(f \circ g)$ (as $I(f \circ g) \subset I(f) \cup I(g)$ and $K(f)\cup K(g) \cap I(f) \cup I(g)=\emptyset$ \cite{d1}) nor with $BU(f \circ g)$ (as $BU(f \circ g)\subset BU(f)\cup BU(g)$ and $BU(f \circ g) \cap K(f) \cup K(g)= \emptyset)$ So, they will intersect with $K(f \circ g).$    

\noindent The following examples illustrate the results established above.

\begin{example}
Consider $f(z)=1+z+e^{-z}$(Fatou map) and  $g(z)=1+z+e^{-z}+ 2\pi i.$ It can be easily checked that $f$ and $g$ satisfies \textbf{Property A}. 
\end{example}
\noindent One may also consider the following functions for illustration of above result. They also satisfy \textbf{Property A}.\\
 $f_1(z)=1+z+e^{z}$  which would be commuting with $g_1(z)=1+z+e^{z}+ 2\pi i$;  $f_2(z)=z+\sin(z)$ which would  be commuting with $g_2(z)=z+ \sin(z)+2\pi.$
\\

\section{Some results for functions of class $\mathcal{B}$}
Now, we will prove some results when the functions belong to  the Eremenko-Lyubich class $\mathcal{B}$. 
These functions do not have a Baker domain; therefore, the escaping set does not intersect with the Fatou set. In addition,  the bungee set of the function does not intersect with a wandering domain in which the iterates tend to $\infty.$ However, functions in this class may have oscillating wandering domain. 

The next result gives a sufficient condition under which the bungee set is contained in the Julia set.
\begin{proposition}
The bungee set for functions of class $\mathcal{B}$ having no oscillatory wandering domain does not intersect with Fatou set, that is, $BU(f)\subset J(f).$ 
\end{proposition}
\begin{proof}
 For a function $f,$ we know that if $BU(f)\cap F(f)\not = \emptyset $, then such components are  oscillatory wandering domains \cite{osb2} and these components are contained inside $BU(f)$. As $F(f)$ has no oscillatory wandering domains, it means, $BU(f)$ does not intersect with $F(f)$. In other words, $BU(f)\subset J(f).$ 
\end{proof}
\begin{remark} \noindent This also means that $BU(f) \subset J(f)$  for functions not necessarily of $\mathcal{B}$ having no oscillatory wandering domain. Also, it is easy to see that $BU(f)$ is completely invariant by the complete invariance of $I(f)$ and $K(f)$. As a result, by the minimality property of Julia set, $\overline{BU(f)}=J(f)$. It is also known that for a function of $\mathcal{B}, J(f)=\overline{I(f)}$ \cite{EL}.
\end{remark}
\noindent The escaping set of a function is in general, neither open, nor closed \cite{e1}. We have similar kind of result for the bungee set.
\begin{proposition}\label{Bug_nop_nclo}
The bungee set for functions of class $\mathcal{B}$ having no oscillatory wandering domain is neither open nor closed. 
\end{proposition}
\begin{proof}
We can easily prove  the complete invariance of bungee set by noting that $BU(f)= \mathbb{C} \backslash (I(f) \cup K(f))$. Now, if we assume that bungee set is  closed then it will be a closed completely invariant set. This implies that $J(f) \subset BU(f)$ (because $J(f)$ is the smallest closed completely invariant set) which is a contradiction as $J(f)$ contains periodic points. Furthermore, if bungee set is open, it will imply that it has interior points which is also not possible (because $BU(f)\subset J(f)$ and bungee set having interior  means Julia set also has interior which is not possible for an entire function $f$ for which $F(f)\neq\emptyset$ \cite{beardon}). Hence, bungee set is neither open nor closed.
\end{proof}
The following result provides equality of bungee set, filled-in Julia set and escaping set of two functions in which one function is periodic translate of the other. Note that the two functions are non-commuting.
\begin{theorem}
Functions which are periodic translate of each other  have equal bungee sets, escaping sets and filled-in Julia sets.
\end{theorem}
\begin{proof}
We consider two functions $f$ and $g$ where $g$ satisfies $g(z)=f(z)+C.$  Here, $C$ is the period of $f$ (more generally, we can take $g=f^l(z)+C$, for some $l\in\N$). It can be easily seen that $g^{n}(z)=f^{n}(z)+C$. Again, it is easy to show that if a point is bounded under $f^n,$ it is also bounded under $g^n.$ Similarly, if a point  is escaping under $f^n,$ it is also escaping under $g^n$. Hence, $K(f)=K(g), I(f)=I(g)$ and  $BU(f)=BU(g)$.
\end{proof}
Such functions can be viewed as almost commuting type with a constant shift.

\section{Illustrations:}
We now illustrate  $BU(f \circ g)\subset BU(f)\cup BU(g)$ with some examples. Our first example is commuting rational functions. \\
Consider the functions $f(z)=z^2$ and $g(z)=1/z^2$. Observe that $f \circ g(z)=1/z^4$. $BU(f)= \emptyset$, $BU(g)=\{z: |z|<1 \cup |z|>1\}$ and $BU(f \circ g)=\{z:|z|<1 \cup |z|>1\}.$ Clearly, $BU(f \circ g)\subset BU(f)\cup BU(g)$.
\\
Our next example is motivated from  \cite{tw}. Let us take the case where $f= p(z)\cdot e^{h(z)}$ where $p(z) , \ h(z)$ are polynomials and $g(z)=a \cdot f(z)$ for some $a \in  \mathbb{C}\setminus \{0\}$. For simplicity, let us take
$f(z)=z\cdot e^{a^{(k+1)} \cdot z^k},\ k \in \mathbb{N}$ \&
$g(z)=a \cdot z \cdot e^{a^{(k+1)} \cdot z^k}$.\\
$f \circ g(z) =f(g(z))=g(z)\cdot e^{a^{(k+1)} \cdot g^k(z)}=aze^{a^{(k+1)}z^k}e^{a^{(2k+1)}z^ke^{ka^{(k+1)}z^k}}$\\
$g \circ f(z) =g(f(z))=a \cdot f(z)\cdot e^{a^k \cdot f^k(z)}= aze^{a^kz^k}e^{a^{2k}z^ke^{ka^kz^k}}.$ Now, for $f \circ g = g \circ f,$ we need to have $a^{2k+1}=a^k \implies a^k(a^{k+1}-1)=0,$ which means $a$ is $(k+1)$-root of unity except `1' itself because that's trivial. 
\\
Now, for the simplest case let us take: $k+1=2 \implies a =\pm 1$ Since $a=1$ is a trivial case so let us take $a=-1$, $f(z)=ze^{-z^2};\ g(z)=-ze^{-z^2}$. 
 Observe that  $z=0$ is an asymptotic value of $f$ and $g$. It is also  a fixed point of $f(z), \ g(z)$ since $f(0)=0$ and $g(0)=0$. Now for the multiplier $\lambda$, $f'(z)=e^{-z^2}(1-2z^2), f'(0)=1$ which is an indifferent fixed point. And $1^1=1$ means $0\in J(f)$ or $J(g)$ (since both of them are equal). As $\infty$ is a special case and belongs to $J(f)$ so, this proves  that a sequence which tends to  $\infty$ if it asymptotes to a finite asymptotic value which is a member of the Julia set of $f$. 
\\
Now some illustrations for the bungee set:\\
Let $f(z)=ze^{z^2};\ g(z)=-ze^{z^2}.$ This implies $f^2=ze^{z^2}e^{z^2e^{2z^2}} \ \& \ g^2=ze^{z^2}e^{z^2e^{2z^2}}.$
$f \circ g = -ze^{z^2}e^{z^2e^{2z^2}} = g \circ f$. And $(f\circ g)^2=g^4=f^4$(because as $(f \circ g)^2$ is done it does away with the negative term and rest terms are exactly same). So, $BU(f^4)=BU(g^4)=BU((f \circ g)^2)$ which implies $BU(f)=BU(g)=BU(f \circ g)$ (\textbf{NOTE:} If we have $f(z)=ze^{-z^2};\ g(z)=-ze^{-z^2}$, they also commute). We may also  take $f(z)=ze^{a^{(k+1)}z^k}$, $g(z)=aze^{a^{(k+1)}z^k}$ and $a^{(k+1)}=1$ for that matter). And this complies with the union relation with the composition.\\
Now, for the general case: Since $a$ is $k$-th root of unity, means $a^k=1$. So, $f(z)=ze^{a^k z^k}=ze^{z^k}$ and $g(z)=aze^{z^k}$. Now $g^k(z)=a^k  f^k(z)=f^k(z)$ implies $BU(f)=BU(g)$. Also, $(f \circ g)^{k}(z)=a^{k}f^{2k}(z)=f^{2k}(z)$. So, $BU(f \circ g)=BU(f)=BU(g)$. Analogously, if we take $f(z)=ze^{a^{(k+1)}z^k}$ and $\ g(z)=af(z)=aze^{a^{(k+1)}z^k}$, where again $a^{k+1}=1$, we can simplify it further as $f(z)=ze^{az^k}$ and $g(z)=aze^{az^k}$, now $g^2(z)=af(af(z)))=a^2f^2(z)$. By induction,  $g^{k+1}(z)=a^{(k+1)}f^{k+1}(z)$. Also, $f \circ g(z)= af^2(z)=g \circ f(z)$ and $(f \circ g)^2(z)=a^2f^4(z)$. Proceeding by induction, $(f \circ g)^{k+1}(z)=a^{(k+1)}f^{2(k+1)}(z)=f^{2(k+1)}(z)$. Here also, we can see that $BU(f\circ g)= BU(f)=BU(g)$.\\
Thus, functions which are of the form $f(z)=ze^{a^k z^k}$ where $a$ is $k$-th root of unity or $f(z)=ze^{a^{k+1} z^k}$, where $a$ is $(k+1)$-th root of unity and $g=af$  have the property that their bungee sets are same. Moreover, their escaping set and repelling periodic point sets are also the same. $J(f)=J(g)$ also holds true  because we are taking  $f= p(z)e^{g(z)}$ as discussed in \cite{tw}. 



\begin{thebibliography}{00}
\bibitem{baker1} I. N. Baker, Limit functions and sets of non-normality in iteration theory, Ann. Acad. Sci. Fenn. Ser. A. I. Math. \textbf{467} (1970), 1-11. 
\bibitem{baker2} I. N. Baker, Wandering domains in the iteration of entire functions, Proc. London Math. Soc. \textbf{49} (1984), 563-576.
\bibitem{beardon} A. F. Beardon, \emph{Iteration of rational functions}, Springer Verlag, (1991).
\bibitem{bergweiler} W. Bergweiler, Iteration of meromorphic functions, Bull. Amer. Math. Soc. \textbf{29} (1993), 151-188.
\bibitem{e1} A. E. Eremenko, On the iteration of entire functions, Ergodic Theory and Dynamical Systems, Banach Center Publications \textbf{23}, Polish Scientific Publishers, Warsaw, (1989), 339-345.
\bibitem{EL} A. E. Eremenko and M. Yu. Lyubich, Dynamical properties of some classes of entire functions, Ann. Inst. Fourier, Grenoble, \textbf{42} (1992), 989-1020.
\bibitem{F2}  P. Fatou, Sur l'it\'eration des fonctions transcendantes Enti\`eres, Acta Math. {\bf 47} (1926), no.~4, 337-370. 
\bibitem{martin} A. Hinkkanen and G. J. Martin, The dynamics of semigroups of rational functions I, Proc. London Math. Soc. (3) \textbf{73} (1996), 358-384.

\bibitem{Hua} X. H. Hua, C. C. Yang, \emph{Dynamics of transcendental functions}, Gordon and Breach Science Pub. (1998).
\bibitem{hua1} X. Hua and X. Wang,  Dynamics of permutable transcendental entire functions, Acta Math. Vietnam. \textbf{27} (2002), 301-306.
\bibitem{dinesh1} D. Kumar and S. Kumar, The dynamics of semigroups of transcendental entire functions I,  Indian J. Pure Appl. Math. \textbf{46} (2015), 11-24.
\bibitem{dk} D. Kumar, On escaping sets of some families of entire functions and dynamics of composite entire functions, Math. Student, \textbf{84} (2015), No. 1-2, 87-94.
\bibitem{dinesh3} D. Kumar and S. Kumar, The dynamics of semigroups of transcendental entire functions II,  Indian J. Pure Appl. Math. \textbf{47} (2015), 409-423.
\bibitem{d1} R. Kaur and D. Kumar, Results on escaping set of an entire function and its composition,  Indian J. Pure Appl. Math. \textbf{52}(2021), 79-86.
\bibitem{tw} T. W. Ng, Permutable entire functions and their Julia sets, Math. Proc. Camb. Phil. Soc.\textbf{131} (2001), 129-138.
\bibitem{osb2} J. W. Osborne and D. J. Sixsmith, On the set where the iterates of an entire function are neither escaping nor bounded, Ann. Acad. Sci. Fenn. Math.  \textbf{41} (2016), 561-578.

\bibitem{marti2020wandering} D. M. Pete  and M. Shishikura, Wandering domains for entire functions of finite order in the Eremenko-Lyubich class, Proc. London Math. Soc. \textbf{120}  (2020), 155-191. 

\bibitem{poon1} K. K. Poon, Fatou-Julia theory on transcendental semigroups, Bull. Austral. Math. Soc. \textbf{58} (1998), 403-410.

\bibitem{rs1}  P. J. Rippon and G. M. Stallard, On questions of Fatou and Eremenko, Proc. Amer.
Math. Soc. \textbf{133} (2005), no. 4, 1119-1126.


\bibitem{sch} D. Schleicher, Dynamics of entire functions,  Lecture Notes in Mathematics  ((LNMCIME,volume 1998)) (2010), 295-339.

\bibitem{aps} A. P. Singh, On bungee sets of composite transcendental entire functions,  arXiv:2006.00208v1[math.CV] (2020) , pp. 1-7.
\end{thebibliography}
 \end{document}